\documentclass [12pt] {article}
\usepackage {amsmath}
\usepackage[latin1]{inputenc}
\usepackage {amsfonts, amssymb,amsthm}
\usepackage {latexsym}
\newtheorem {theorem} {Theorem} [section]
\newtheorem {lemma} {Lemma} [section]

 \newtheorem{preremark}{Remark}[section]
  \newenvironment{remark}%
    {\begin{preremark}\rm}{\end{preremark}}
     \newtheorem{preremark1}{Example}[section]
    {\begin{preremark1}\rm}{\end{preremark1}}
         \newtheorem{preremark2}{Definition}[section]
  \newenvironment{definition}%
    {\begin{preremark2}\rm}{\end{preremark2}}

\begin{document}

\title{Anomalous diffusion in polymers: long-time behaviour}

\author{Dmitry A. Vorotnikov}

\date{}

\maketitle

\begin{center}

CMUC (Centro de Matemática da Universidade de Coimbra)

Apartado 3008, 3001 - 454 Coimbra, Portugal

mitvorot@mat.uc.pt

\end{center}





\begin{abstract}
We study the
Dirichlet boundary value problem for viscoelastic diffusion in polymers. We show that its weak solutions generate a dissipative semiflow. We construct the minimal trajectory attractor and the global attractor for this problem.
\end{abstract}



\section {Introduction}

\renewcommand{\theequation}{\arabic{section}.\arabic{equation}}

\numberwithin{equation}{section}

\newcommand {\R} {\mathbb{R}}
\newcommand {\E} {\mathbf{E}}
\def\be{\begin{equation}}
\def\ee{\end{equation}}
\def\fr#1#2{\frac{\partial #1}{\partial #2}}

The concentration behaviour for diffusion of penetrant liquids in
poly\-mers cannot always be described by the Fickian diffusion
equation \be\fr u t = div (D(u) \nabla u),\ee where $u=u(t,x)$ is
the concentration, which depends on time $t$ and the spatial point
$x$, and $D(u)$ is the diffusion coefficient. The phenomena
running counter to (1.1) include \textit{case II diffusion},
\textit{sorption overshoot}, \textit{literal skinning, trapping
skinning} and \textit{desorption overshoot} \cite{chn1, ed,ed2,
cc, tw,tw1, var, wit}. There is a number of approaches which
explain these non-Fickian properties of polymeric diffusion. They
have much in common: they are usually based on taking into account
the viscoelastic nature of polymers (cf. \cite{lee} and references
therein) and on the possibility of glass-rubber phase transition
(see e.g. \cite{wit} with some review). We are going to study the
model which is due to Cohen et al. \cite{chn0,chn1,chn3}. The
Fickian diffusion equation is replaced by the system \be\fr u t =
D \Delta u + E\Delta\sigma,\ee \be \fr {\sigma} t +
\beta(u,\sigma)  \sigma = \mu u+ \nu \fr u t. \ee Here the second
variable $\sigma(t,x)$ is introduced (it is called
\textit{stress}),  $D$ and $E$ are the diffusion and
stress-diffusion coefficients, resp., $\mu$ and $\nu$ are
non-negative constants, and the scalar function $\beta$ is the
inverse of the relaxation time, for instance, $\beta$ can be
\cite{chn1} taken in the following form: \be\beta=\beta(u)= \frac
1 2 (\beta_R+ \beta_G)+ \frac 1 2 (\beta_R- \beta_G) \tanh (\frac
{u- u_{RG}} {\delta})\ee where $\beta_R,\beta_G, \delta, u_{RG}$
are positive constants, $\beta_R>\beta_G$\footnote{Formula (1.4)
describes the following peculiarities of the processes under
consideration. The polymer network in the glassy state (low
concentration area) is severely entangled, so $\beta$ is
approximately equal to some small $\beta_G$. In the high
concentration areas the system is in the rubbery state: the
network disentangles, so the relaxation time is small, and its
inverse is close to $\beta_R>\beta_G$. The glass-rubber phase
transition occurs near a certain concentration $u_{RG}$. However,
we assume that $\beta$ also depends on stress, cf.
\cite{am2,ed2,var}.}.

Well-posedness issues for initial-boundary value problems for systems of viscoelastic diffusion equations  have been studied in \cite{am1,am2,bei,diss,var,nova}, see \cite{var,iter} for brief reviews. These results include investigation of system (1.2),(1.3) and more general settings (diffusion with variable coefficients). Let us only recall the main results on global (in time) solvability:
strong solutions exist globally for $\nu=0$ and $D=E$ in the one-dimensional case \cite{am1}, and for suitable non-constant stress-diffusion coefficient, but not for all initial and boundary data \cite{bei}; global existence of weak solutions for the Dirichlet and Neumann problems in the general setting with variable coefficients in the multidimensional case is shown in \cite{var} and \cite{nova}, resp., without restrictions on the initial and boundary data.

Let us also mention here a study of a system obtained from (1.2)--(1.3) by some simplification, in \cite{riv}, and paper \cite{hu}, which touches upon some long-time behaviour issues for a free boundary problem for a polymeric diffusion model based on Fick's law. A result on long-time behavior (not in the "attractor framework") of the general second boundary value problem can be found in \cite{nova}.

In this work we are interested in the long-time behaviour of the solutions to Dirichlet initial-boundary value problem for system (1.2),(1.3). We show that the weak solutions generate a semiflow on a suitable phase space with $L_2$-topology (this means that there is a unique solution for any data from the phase space, and the solution semigroup is continuous in $t$ and $x$). However, it is not clear whether this semigroup is asymptotically compact, and the phase space is not complete, so this impedes proving of existence of the usual global attractor for this semigroup in this phase space. A possible way out is to use the concept of minimal trajectory attractor. Thus, we construct a minimal trajectory attractor, which generates some generalized global attractor for the weak solutions of the problem in the (completed) phase space.

The theory of trajectory attractors was created by G. Sell, M. Vishik and V. Chepyzhov \cite{vis2,cvbook,sel},  in order to construct an attractor to weak solutions of the 3D Navier-Stokes equation. A generalization of this approach with a related notion of minimal trajectory attractor may be found in \cite{gruy,jmfm}; it is applicable when the system lacks continuity properties or invariance of the trajectory space with respect to time shifts. In both theories, the trajectory attractor generates some generalized global attractor in the phase space. This
 global attractor has many usual properties of attractors, but its invariance may be shown only under additional conditions (see \cite[Section 4.2.7]{gruy}).
  The idea of trajectory attractor was slightly criticized by Sir J. Ball \cite{ball}, for the evolution of the original system is not explicitly involved in its definition. However, an example \cite[Remark 4.2.13]{gruy} shows that the minimal trajectory attractor (and the corresponding global attractor) can well characterize the long-time behaviour of the system, even when the usual global attractor does not exist.
That example also illustrates that such situations may appear, in particular, for problems with uniqueness of solutions (i.e. when there exists a solution semigroup). It is a rather unexpected fact because the original theory of trajectory attractors was developed for the problems where the uniqueness
is not proved or is absent\footnote{Trajectory attractors for problems with uniqueness were investigated in \cite{cvbook} only as an intermediate step on the way to usual global attractors of semigroups.}.  Similarly, in this paper we apply the theory of minimal trajectory attractors to a problem with uniqueness. However, in our case, we cannot insist that there is no attractor of the semigroup (semiflow), since this is unknown.

Our paper
is organized
in the following way. In Section 2, we introduce the required function spaces.
In Section 3, we give a weak formulation of the initial-boundary value problem for system (1.2)-(1.3) with existence, uniqueness and regularity results (Theorem 3.2, Remark 3.4).
In Section 4, we recall the basic issues of the classical and minimal trajectory attractor theories, and construct a dissipative semiflow generated by weak solutions of (1.2)-(1.3) (Theorem 4.10). In Section 5, we show that the \textit{trajectory space} generated by the weak solutions possesses a minimal trajectory attractor and a global attractor (Theorem 5.2).

\section {Function spaces and related notations}
$L_p (\Omega) $, $W_p^{m} (\Omega)
$, $H^{m} (\Omega) =$ $W_2^{m} (\Omega) $ $(m\in\mathbb{Z}, 1 \leq
p \leq \infty)$, $H^{m}_0 (\Omega) = \stackrel{\circ}{W}{}_2^{m} (\Omega) $ $(m \in \mathbb{N})$ are, as usual, Lebesgue and Sobolev
spaces of functions defined on a bounded open set (domain) $\Omega\subset
\R^n$, $n\in \mathbb{N}$. The scalar product and the Euclidian norm in
$L_2(\Omega)^k=L_2(\Omega, \R^k)$ are denoted by $(u,v)$ and
$\|u\|$, respectively ($k$ is equal to $1$ or $n$). In $H^1_0(\Omega)$, we use the following scalar product and norm:
$(u,v)_1=(\nabla u, \nabla v), \|u\|_1=\|\nabla u\|$.
We recall Friedrichs' inequality \be \|u\|\leq K_\Omega\|u\|_1.\ee


Let $L^1_2(\Omega)$ denote the topological subspace of $L_2(\Omega)$ consisting of functions from $H^1_0(\Omega)$.

The Laplace operator $\Delta: H^1_0(\Omega)\to H^{-1}(\Omega)$ is an isomorphism.
Therefore \be\Delta^{-1}: H^{-1}(\Omega)\to H^{1}_0(\Omega)\ee is also an isomorphism.
Set $X=X(\Omega)=\Delta^{-1}(H^{1}_0(\Omega))$. The scalar product and norm in $X$ are $(u,v)_X = (\Delta u, \Delta v)_1$, $\|u\|_X = \|\Delta u\|_1$.

As usual, we identify the space $H^{-1}(\Omega)$ with the space of linear continuous functionals on
$H_0^1(\Omega)$ (the dual space). The value of a functional from
$H ^ {-1}(\Omega) $ on an element from $H_0^1(\Omega) $ is denoted
by $ \langle\cdot, \cdot\rangle $ (the "bra-ket" notation).  The scalar product and norm in $H^{-1}(\Omega)$ are $(u,v)_{-1} = (\Delta^{-1} u, \Delta^{-1} v)_1$, $\|u\|_{-1} = \|\Delta^{-1} u\|_1$. Note that
\be (u,\Delta v)_{-1}= -\langle u, v\rangle,\ u\in H ^ {-1}(\Omega), v \in H_0^1(\Omega). \ee

The symbols $C (\mathcal{J}; E) $, $L_2
(\mathcal{J}; E) $ etc. denote the spaces of continu\-ous, quadratically integrable etc. functions on an interval
$\mathcal{J}\subset \mathbb {R} $ with
values in a Banach space $E $.

Let us remind that a pre-norm in the Frechet space $C ([0, +
\infty); E) $ may be defined by the formula
$$ \| v \| _{C ([0, +\infty);E)} =\sum\limits _ {i=1} ^ {+ \infty} 2 ^ {-i} \frac {\|v \| _ {C ([0, i]; E)}} {1 + \| v \| _ {C ([0, i]; E)}}. $$

If $E$ is a function space ($L_2(\Omega), H^m(\Omega)$ etc.), then
we identify the elements of $C (\mathcal{J}; E) $,
$L_2(\mathcal{J};E)$ etc. with scalar functions defined on
$\mathcal{J}\times \Omega$ according to the formula
$$u(t)(x)=u(t,x),\, t\in \mathcal{J}, x\in \Omega.$$

We shall also use the function space ($T$ is a positive number):
$$ W=W(\Omega, T) = \{u\in L_2 (0, T; H^1_0 (\Omega)), \
u' \in L_2 (0, T;
 H ^ {-1} (\Omega)) \},$$ $$\|u\|_{W}=\|u\|_{L_2 (0, T; H^1_0(\Omega))}+\|u'\|_{L_2 (0, T; H ^ {-1}(\Omega))};$$

\cite[Corollary 2.2.3]{gruy} implies continuous embedding $W\subset
C([0,T];L_2(\Omega))$. Moreover, \be \left\langle  v',v\right\rangle=\frac 1 2\frac {d}{dt}\|v\|^2,\ v\in W.\ee

We also denote

$$ W_{loc}(\Omega,+\infty)= \{u\in C ([0, +\infty); L_2 (\Omega)), \
u|_{[0,T]} \in W(\Omega, T)\ \forall T>0\},$$

$$H^1_{loc}(0,+\infty;H^1_0(\Omega)) = \{u\in C ([0, +\infty); H^1_0(\Omega)), \
u|_{[0,T]} \in H^1(0,T;H^1_0(\Omega))\ \forall T>0\}.$$

We use the notation $|\cdot|$ for the absolute value of a number and for the Euclidean
norm in $\R^n$.

The symbols $C$ stands for various positive constants.

Note that  \be\langle u,v \rangle= ( u,v ),\ u\in L_2(\Omega), v \in H^1_0(\Omega)\ee and, since $H^{-1}(\Omega) \subset L_2(\Omega)$, one has
\be \|u\|_{-1}\leq C\|u\|.\ee

\section {Basic properties of the boundary value problem}

We study the diffusion of
a penetrant in a polymer filling a bounded domain $\Omega\subset\mathbb{R}^n$,
$n\in \mathbb{N}$, which is described by the following
boundary value problem:

\be \fr u t = D\Delta u + E \Delta \sigma,\ (t,x)\in [0,\infty)\times \Omega,\ee \be \fr {\sigma} t
+ \beta_0 (u,\sigma) \sigma = \mu u + \nu \fr u t,\ (t,x)\in [0,\infty)\times \Omega, \ee
\be u(t,x)=\varphi(x),\ (t,x)\in [0,\infty)\times \partial\Omega. \ee


Here $u=u(t,x):[0,\infty)\times \overline{\Omega} \to \R$ is the
unknown concentration of the penetrant (at the spatial point $x$
at the moment of time $t$), $\sigma=\sigma(t,x):[0,\infty)\times
\overline{\Omega} \to \R$ is the unknown stress,
$\varphi:\partial\Omega \to \R$ is a given boundary condition,
$\mu,\nu$, $D$ and $E$ are positive constants\footnote{The case
$\mu=0$ (''the Maxwell model'' \cite{dur,wit}) is admissible as
well.}, $ \beta_0: \R^{2}\to \R$ is a given function, $
\beta_R\geq\beta_0(u,\sigma)\geq  \beta_G>0$;
$\Omega\subset\mathbb{R}^n$ is supposed to be a bounded open set
such that $ X(\Omega)\subset W^1_{p_0}(\Omega)$ for some $p_0>2$
(cf. \cite{var}). For definiteness, we assume that the boundary
condition for the stress is also prescribed: \be
\sigma(t,x)=\phi(x),\ (t,x)\in [0,\infty)\times \partial\Omega.
\ee

Equation (3.2) yields the following relation between the boundary conditions $\phi$ and $\varphi$:
\be  \beta_0 (\varphi,\phi) \phi = \mu \varphi,\ x\in \partial\Omega. \ee

W.l.o.g. the functions $\varphi$ and $\psi$ are defined on $\overline{\Omega}$.

We assume that the functions $ \beta_0, \phi$ and $\psi$ are $C^2-$ smooth.
Moreover, for simplicity, let $ \beta_0(u,\sigma)\equiv \beta_\infty\geq\beta_G$ for large $|u|+|\sigma|$ (this assumption is admissible in the considered model, see Remark 3.6 below).


Set $$v(t,x)=u(t,x)-\varphi(x), \varpi(t,x)=\sigma(t,x) - \phi(x),$$
$$\beta(x,v,\varpi)=\beta_0(v+\varphi(x),\varpi+\phi(x)), $$
$$h(x)=D\Delta \varphi (x)+ E \Delta \phi(x),$$
$$g(x,v,\varpi)=\mu \varphi(x)-\beta_0(v+\varphi(x),\varpi+\phi(x))\phi(x).$$

Then we can rewrite (3.1)-(3.4) in the following form:
\be \fr v t = D\Delta v + E \Delta \varpi+h,\ee \be \fr {\varpi} t
+ \beta (x,v,\varpi) \varpi = g(x,v,\varpi)+\mu v + \nu \fr v t, \ee
\be v|_{\partial\Omega}=\varpi|_{\partial\Omega}= 0. \ee

This form of the studied problem will be useful below.
However, in order to set the problem finally, we need the variable
$\tau(t,x)=\varpi(t,x)-\nu v(t,x).$
We denote $$d=D+\nu E,$$  $$\gamma\left(x,v,\tau\right)=\mu v-\beta(x,v,\tau+\nu v)\tau-\nu\beta(x,v,\tau+\nu v)v+g(x,v,\tau+\nu v).$$
Then problem (3.6)-(3.8) becomes
\be \fr v t = d\Delta v + E \Delta \tau+h,\ee \be \fr {\tau} t = \gamma(x,v,\tau),\ee \be v|_{\partial\Omega}=\tau|_{\partial\Omega}= 0. \ee
It can be completed with the initial condition:

\be v(0,x)=v_0(x), \ \tau(0,x) = \tau_0(x),\ x\in\Omega. \ee

\begin{definition}  A pair of
functions $(v,\tau)$ from the class \be v \in W_{loc}(\Omega,+\infty), \tau\in
H^1_{loc}(0,+\infty;H^1_0(\Omega))\ee is a {\it weak} solution to problem (3.9)-(3.11) if equality (3.9) holds in the space $H^{-1}(\Omega)$ for a.a.
$t\in(0,+\infty)$, and (3.10)  holds in $H^{1}(\Omega)$ a.e. on
$(0,+\infty)$. \end{definition}

\begin{theorem} Given $v_0 \in L_2(\Omega)$ and $\tau_0\in H^1_0(\Omega)$, there exists a unique weak solution to problem (3.9)-(3.11) which belongs to (3.13) and satisfies (3.12).\end{theorem}

Note that (3.12) makes sense due to the embeddings $W(\Omega,T)\subset C([0,T];L_2(\Omega)),$\\ $ H^1(0,T;H^1_0(\Omega))\subset
C([0,T];H^1_0(\Omega)), T>0.$

\begin{proof} Let us calculate the gradient of $\gamma(x,v,\tau)$:
$$\nabla\gamma (x,v,\tau) $$ 
$$= \mu \nabla v- \beta (x,v,\tau+\nu v) \nabla \tau -\fr{\beta}{x} (x,v,\tau+\nu v) \tau$$ $$-\fr{\beta}{v} (x,v,\tau+\nu v) \tau\nabla v $$ $$-\nu\fr{\beta}{\varpi} (x,v,\tau+\nu v) \tau\nabla v-\fr{\beta}{\varpi} (x,v,\tau+\nu v) \tau\nabla\tau$$ $$- \nu \beta (x,v,\tau+\nu v) \nabla v -\nu \fr{\beta}{x} (x,v,\tau+\nu v) v$$ $$-\nu\fr{\beta}{v} (x,v,\tau+\nu v) v\nabla v -\nu^2\fr{\beta}{\varpi} (x,v,\tau+\nu v) v\nabla v-\nu\fr{\beta}{\varpi} (x,v,\tau+\nu v) v\nabla\tau$$ $$+ \fr{g}{x}(x,v,\tau+\nu v) +\fr{g}{v}(x,v,\tau+\nu v)  \nabla v +\nu\fr{g}{\varpi}(x,v,\tau+\nu v)  \nabla v +\fr{g}{\varpi} (x,v,\tau+\nu v)\nabla\tau.$$

Compactness of $\Omega$ gives boundedness of $\phi$ and $\varphi$. Hence, the function $\beta(x,v,\varpi)$ is equal to $\beta_\infty$ for large  $|v|+|\varpi|$.
This implies that the gradient $\nabla\gamma(x,v,\tau)$ can be expressed in the form $$h_1(x,v,\tau)\nabla v+h_2(x,v,\tau)\nabla \tau+\theta(x,v,\tau),$$ where $$ |h_1(x,v,\tau)|+ |h_2(x,v,\tau)|\leq K_h,$$
$$ |\theta(x,v,\tau)|\leq K_\theta(|v|+|\tau|+1)$$
with some constants $K_h, K_\theta$.

Therefore, by \cite[Theorem 3.1]{var}, there is a pair of functions $(v,\tau)$ from (3.13) which satisfies (3.9), (3.12) and  \be \Delta\fr {\tau} t = \Delta[\gamma(x,v,\tau)],\ee in $H^{-1}(\Omega)$ a.e. on
$(0,+\infty)$
(more precisely, that theorem gives existence of solutions on finite time intervals, but then the solution on the  positive semi-axis can be constructed step by step in a standard way, see \cite{iter}).

Let us show that $(v,\tau)$ satisfies (3.10). Since $\fr {\tau} t\in H^1_0(\Omega)$ for a.a. $t>0$, it suffices to prove that $\gamma(x,v,\tau)\in H^1_0(\Omega)$ for a.a. $t>0$. Observe that $\gamma(x,0,0)\equiv 0$ due to (3.5). The above representations of $\gamma(x,v,\tau)$ and $\nabla \gamma(x,v,\tau)$ imply $\gamma(x,v(t,x),\tau(t,x))$ $\in H^1(\Omega)$ (for a.a. $t>0$). Let $v_m,\tau_m$ be sequences of smooth functions with compact supports in $\Omega$, $v_m \to v(t),\tau_m \to \tau(t)$ in $H^1(\Omega)$.  Then $\gamma(\cdot,v_m,\tau_m)\in H^1_0(\Omega)$. It remains to observe that  $\gamma(\cdot,v_m,\tau_m)\to \gamma(\cdot,v(t),\tau(t))$ weakly in $H^1(\Omega)$. Really,  by Krasnoselskii's theorem \cite{kras,skry} on
continuity of Nemytskii operators we have $$\gamma(\cdot,v_m,\tau_m)\to \gamma(\cdot,v(t),\tau(t)), h_1(\cdot,v_m,\tau_m)\to h_1(\cdot,v(t),\tau(t)),$$ $$h_2(\cdot,v_m,\tau_m)\to h_2(\cdot,v(t),\tau(t)), \theta(\cdot,v_m,\tau_m)\to \theta(\cdot,v(t),\tau(t))$$ strongly in $L_2(\Omega)$. Moreover, due to boundedness of $h_1$ and $h_2$, w.l.o.g. we may assume that $h_1(\cdot,v_m,\tau_m)\to h_1(\cdot,v(t),\tau(t))$, $h_2(\cdot,v_m,\tau_m)\to h_2(\cdot,v(t),\tau(t))$ *-weakly in $L_\infty(\Omega)$. But $\nabla v_m\to  \nabla v(t)$, $\nabla \tau_m\to  \nabla\tau(t)$ strongly in $L_2(\Omega)$. Therefore, $$h_1(\cdot,v_m,\tau_m)\nabla v_m\to h_1(\cdot,v(t),\tau(t))\nabla v(t),h_2(\cdot,v_m,\tau_m)\nabla\tau_m\to h_2(\cdot,v(t),\tau(t))\nabla \tau(t)$$ weakly in $L_2(\Omega)$.

\begin{remark} It seems that a more profound reasoning of this kind proves that for any Nemytskii operator $\mathcal{N}:H^1(\Omega)\to H^1(\Omega)$ one has $\mathcal{N}(\xi)- \mathcal{N}(0)\in H^1_0(\Omega)$ for all $\xi\in H^1_0(\Omega)$.\end{remark}

It remains to prove uniqueness. If $\varpi=\tau+\nu v$, then $\varpi\in W_{loc}(\Omega,+\infty)$, and the pair $(v,\varpi)$ satisfies (3.6), (3.7). It suffices to show uniqueness of the pair $(v,\varpi)$. Let $(v_1,\varpi_1)$, $(v_2,\varpi_2)$ be solutions of (3.6), (3.7) in the class $W_{loc}(\Omega,+\infty)\times W_{loc}(\Omega,+\infty)$ with the same initial conditions. Let us denote $w=v_1 - v_2,$ $ \xi=\varpi_1 -\varpi_2$. Then
\be w' = D\Delta w + E \Delta \xi,\ee \be \xi'
+ \beta (x,v_1,\varpi_1) \varpi_1 -\beta (x,v_2,\varpi_2) \varpi_2= g(x,v_1,\varpi_1)-g(x,v_2,\varpi_2)+\mu w + \nu w' . \ee

Calculate the
$H^{-1}(\Omega)$-scalar product of (3.15) and $\mu w +\nu w'$ for a.a. $t\in (0,\infty)$, and take (2.3) and (2.5) into account:
\be \mu( w',w)_{-1} + \nu( w',w')_{-1}=- \mu D(w, w)  - \mu E(\xi, w) -\nu D\left\langle w', w\right\rangle  - \nu E\left\langle w',\xi\right\rangle.\ee
Take the
"bra-ket"
 of (3.16) and $E\xi$ for a.a. $t\in (0,\infty)$, and remember (2.5):
\be E\left\langle\xi',\xi\right\rangle
+ E (\beta (x,v_1,\varpi_1) \varpi_1 -\beta (x,v_2,\varpi_2) \varpi_2, \xi)=$$ $$ E(g(x,v_1,\varpi_1)-g(x,v_2,\varpi_2), \xi)+\mu E(w, \xi) + \nu E \left\langle w',\xi\right\rangle. \ee
Adding (3.17) and (3.18), and omitting the second term, which is positive, we conclude:
\be \frac \mu 2 \frac{d \|w\|^2_{-1}}{d t}+ \mu D\|w\| ^2 + \frac {\nu D} 2\frac{d \|w\|^2}{d t}  +\frac E 2 \frac{d \|\xi\|^2}{d t}$$ $$\leq E (\beta_0 (v_2+\varphi,\varpi_2+\phi) (\varpi_2+\phi) -\beta_0 (v_1+\varphi,\varpi_1+\phi) (\varpi_1+\phi) , \xi)$$ $$\leq C(\|\xi\|+\|w\|)\|\xi\|.\ee The last inequality follows from boundedness of $\fr {[\beta_0(u,\sigma)\sigma]}{u}$ and $\fr {[\beta_0(u,\sigma)\sigma]}{\sigma}$, Lagrange's theorem and the Cauchy-Buniakowski inequality.

Integration from $0$ to $t$ yields $$ \frac \mu 2 \|w\|^2_{-1}+ \mu D\int\limits_0^t \|w(s)\| ^2\,ds + \frac {\nu D} 2\|w\|^2  +\frac E 2 \|\xi\|^2 $$ $$\leq C\int \limits_0^t(\|\xi(s)\|^2+\|w(s)\|^2)\,ds.$$
Hence, \be \|w(t)\|^2  + \|\xi(t)\|^2\leq C\int \limits_0^t(\|\xi(s)\|^2+\|w(s)\|^2)\,ds,\ t\geq 0.\ee
Thus, by the Gronwall lemma, $\xi\equiv w\equiv 0$. \end{proof}

\begin{remark} Weak solutions $(v(t),\tau(t))$ to (3.9)-(3.11) belong to $H ^ {1}_0 (\Omega) ^2$ for all $t>0$.
It follows from a simple regularity result for reaction-diffusion equations. Consider the problem \be \fr \upsilon t - a\Delta \upsilon =f,\ee \be \upsilon|_{\partial\Omega}= 0, \ee
\be \upsilon(0,x)=\upsilon_0(x),\ x\in\Omega. \ee
Given $a>0$, $\upsilon_0 \in L_2(\Omega)$ and $f\in L_2(0,T; H^{-1}(\Omega))$, there exists a unique weak solution $\upsilon\in W(\Omega, T)$, $T>0$. Since $W(\Omega, T)\subset C([0,T],L_2(\Omega))$, the operator $$\Upsilon: L_2(\Omega)\times L_2(0,T; H^{-1}(\Omega))\to L_2(\Omega),\ \Upsilon(\upsilon_0,f)=\upsilon(t), 0<t \leq T,$$ is well-defined. Since $\tau \in H^1(0,T;H ^ {1}_0 (\Omega)) $, it suffices to apply the following lemma to equation (3.9). \end{remark}
\begin{lemma} The operator $\Upsilon$ transforms  $L_2(\Omega)\times H^1(0,T; H^{-1}(\Omega))$ into $H^1_0(\Omega)$. \end{lemma}  

\begin{proof} W.l.o.g. $a=1$. The solution $\upsilon$ can be considered as the sum $\upsilon_1 +\upsilon_2$ of the weak solutions  to the following problems: \be \fr {\upsilon_1} t - \Delta \upsilon_1 =f,\ee \be \upsilon_1|_{\partial\Omega}= 0, \ee
\be \upsilon_1(0)=-\Delta^{-1}f(0), \ee \be \fr {\upsilon_2} t - \Delta \upsilon_2 =0,\ee \be \upsilon_2|_{\partial\Omega}= 0, \ee
\be \upsilon_2(0)=\upsilon_0+\Delta^{-1}f(0). \ee

It is easy to see that $\upsilon_1(t)=-\Delta^{-1}f(0)+\int\limits_0^t \upsilon_3(s)\, ds$, where $\upsilon_3\in W(\Omega, T)$ is determined by the problem \be \upsilon_3 ' - \Delta \upsilon_3 =f',\ee \be \upsilon_3|_{\partial\Omega}= 0, \ee
\be \upsilon_3(0)=0. \ee

Obviously, $\upsilon_3 \in L_2(0,T; H^{1}_0(\Omega))$, so $\upsilon_1\in H^1(0,T; H^{1}_0(\Omega))\subset C([0,T]; H^{1}_0(\Omega))$. Moreover, see e.g. \cite[Proposition XV.3.5]{cvbook}, $\upsilon_2(t)\in H^{1}_0(\Omega)$.\end{proof}

\begin{remark} There is no essential loss
in generality of the model in assuming that $ \beta_0(u,\sigma)$ is equal to some constant $\beta_\infty$ for large $|u|+|\sigma|$. In fact, physically, $ |u(t,x)|\leq 1$
($u$ is the concentration, so it cannot exceed $100\%$). Let $\varsigma= \sigma -\nu u$. Then (3.2) yields (cf. \cite{var,iter})
$$ \varsigma(t,x)=\varsigma(0,x)\exp\left(-\int\limits_0^t \beta_0(u(\xi,x),\nu u(\xi,x)+\varsigma(\xi,x))\, d \xi\right)$$ $$+\int\limits_{0}^t
\exp\left(\int\limits_t^s \beta_0(u(\xi,x),\nu u(\xi,x)+\varsigma(\xi,x))\, d \xi\right)$$ $$\times \left[\mu -\nu  \beta_0(u(s,x),\nu u(s,x)+\varsigma(s,x))\right]u(s,x)
\, ds.$$ Hence,
\be |\varsigma(t,x)|\leq e^{-t\beta_G}|\varsigma(0,x)|+(\mu+\nu\beta_R)\int\limits_{0}^t
e^{(s-t)\beta_G}
\, ds\leq e^{-t\beta_G}|\varsigma(0,x)|+\frac{\mu+\nu\beta_R}{\beta_G}.\ee
Thus, if $|\varsigma(0,x)|$ is uniformly bounded, $\varsigma$ is also bounded:
\be |\varsigma(t,x)|\leq C, \forall t>0, x\in \Omega.\ee Moreover, (3.33) implies that long-time behaviour of $\varsigma(t,x)$ is bounded  for any $|\varsigma(0,x)|$. Hence, $\varsigma$ (and, therefore, $\sigma$) is bounded for typical regimes which may be observed in reality. Thus, the relaxation time (and its inverse $ \beta_0$) can be experimentally determined\footnote{e.g. in form (1.4).} only for bounded $u$ and $\sigma$, whereas "at infinity" we can choose it at discretion, for instance, we can let $ \beta_0(u,\sigma)\equiv \beta_\infty$ for large $|u|+|\sigma|$.

\end{remark}

\section {Semigroups, semiflows, trajectory attractors and global attractors}

Let us recall some basics of the attractor theory. Let $E$ be a metric space.

\begin{definition} A family of mappings $\mathcal{S}_t:E\to E$, $t\geq 0$, is called a
\textit{semigroup}  if $\mathcal{S}_0$ is the identity map $I$ and
\begin{equation}\mathcal{S}_t\circ \mathcal{S}_s=\mathcal{S}_{t+s}\end{equation} for any
$t,s\geq 0$.
\end{definition}

\begin{definition} A set $P\subset $ $E$ is called \textit{attracting} (for $\mathcal{S}_t$) if for any bounded set
$B\subset E$ and any open neighborhood $W$ of $P$ there exists $h
\geq 0$ such that $\mathcal{S}_t B\subset W$ for all $t\geq h$.
\end{definition} \begin{definition} 
 A set $P\subset $ $E$ is called \textit{absorbing}  (for $\mathcal{S}_t$) if for any bounded set $B\subset E$ there is $h \geq 0 $ such that for all $t \geq h $ one
has $ \mathcal{S}_t B\subset P. $ \end{definition}

\begin{definition} 
 A set $A\subset $ $E$ is called \textit{invariant} (for $\mathcal{S}_t$) if

$$\mathcal{S}_t A= A $$ for any $t \geq 0
$.\end{definition}

\begin{definition} A set $\mathcal{A}\subset E$ is called a \textit{global attractor} (of
$\mathcal{S}_t$) if

i) $\mathcal{A} $ is compact;

ii) $\mathcal{A} $ is invariant for $\mathcal{S}_t$;

iii) $\mathcal{A} $ is attracting for
$\mathcal{S}_t$.
\end{definition}

If there exists a global attractor of
$\mathcal{S}_t$, then it is unique (see e.g. \cite[Corollary 4.1.1]{gruy}).

\begin{definition} A semigroup $\mathcal{S}_t:E\to E$ is called
\textit{dissipative}  if there is a bounded absorbing set.
\end{definition}

\begin{definition} A semigroup $\mathcal{S}_t:E\to E$ is called
\textit{asymptotically compact}  if, for any bounded sequence $y_m \in E$ and any sequence of numbers $t_m\to$ $+\infty$, the sequence $\mathcal{S}_{t_m}(y_m)$ contains a converging subsequence.
\end{definition}

\begin{definition} A semigroup $\mathcal{S}_t:E\to E$ is called a
\textit{semiflow}  if the map $$(t, y)\mapsto \mathcal{S}_t(y)$$ is continuous from $[0,+\infty)\times E$ to $E$.
\end{definition}

The next result follows e.g. from \cite[Theorem 3.3, Corollary 4.3]{ball}:
\begin {theorem}
A semiflow $\mathcal{S}_t:E\to E$ has a global attractor $ \mathcal {A}\subset E $ if and only if it is dissipative and asymptotically compact. If $E$ is connected, then $ \mathcal {A}$ is a connected set. \end {theorem}

In order to describe the dynamics of weak solutions for problem (3.9)-(3.11), one may put $E=L_2(\Omega) \times L_2^1(\Omega)$, and define the semigroup $S_t:E\to E$ in the standard way: if $y=(v_0,\tau_0)$, and $(v,\tau)$ is the corresponding weak solution of (3.9)-(3.12), then we set $S_t(y)=(v(t),\tau(t))$. Since weak solutions  belong  to  $C([0,T]; L_2(\Omega)) \times C([0,T]; H^1(\Omega))$ for all $T>0$, the map $t\mapsto S_t(y)$ is continuous for each $y\in E$. On the other hand, a reasoning similar to the proof of uniqueness in Theorem 3.2 shows that the map $S_t:E\to E$ is continuous uniformly with respect to $t\in [0,T]$ for all $T>0$.  Then the map $(t, y)\mapsto S_t(y)$ is continuous, and, taking into account Lemma 5.1 (see below), we arrive at

\begin{theorem} $S_t:E\to E$ is a dissipative semiflow. \end{theorem}

However, it seems to be very hard or impossible to establish asymptotic compactness of this semiflow since the space $E$ is not complete (and this is usually important, cf. \cite{moi}). If we try to change the phase space $E$ and to find so-called $(E,F)$-attractors \cite{vib,gruy}, i.e. attractors which attract in the topology of some space $F$, which is weaker than the one of $E$, then we need continuity of $S_t$ in this weaker topology, which is not clear. Moreover, it seems (due to the properties of the generalized attractor which we construct below) that the attractor can contain non-differentiable elements of $L_2(\Omega)$, and then Remark 3.4 implies that it cannot be invariant. Thus, we are going to use the theory of trajectory attractors, so let us briefly describe the required notions.

Let $E $ and $E_0 $ be Banach spaces, $E\subset E_0 $, $E $ is reflexive.
Fix some set
$$\mathcal {H} ^ + \subset C ([0, + \infty); E_0) \cap L_\infty (0, + \infty; E) $$
of solutions (strong, weak, etc.) for any given autonomous differential equation or boundary value problem. Hereafter, the set $ \mathcal {H} ^ + $ will be
called the \textit{trajectory space} and its elements will be
called  \textit{trajectories}. Generally speaking, the nature of $ \mathcal {H} ^ + $ may be
different from the just described one.

\begin{definition} A set $P\subset $ $C ([0, + \infty ); E_0) \cap L_\infty (0,
+ \infty; E) $ is called \textit{attracting} (for the trajectory
space $ \mathcal {H} ^ + $) if for any set $B\subset \mathcal {H} ^ + $ which is bounded in $ L_\infty (0, +
\infty; E) $, one has
$$\sup\limits _ {u\in B} ^ {} \inf\limits _ {v\in P} ^ {} \|T (h) u-v \| _{C ([0, +\infty);E_0)} \underset {h\to\infty} {\to} 0. $$
\end{definition}
Here $T(h)$ stands for the  translation (shift) operators,
$$ T (h) (u) (t) =u (t+h). $$ Note that $T(h)$ is a semiflow on $C ([0, + \infty ); E_0)$.

\begin{definition} 
 A set $P\subset $ $C ([0, + \infty ); E_0) \cap L_\infty (0, +
\infty; E) $ is called \textit{absorbing} (for the trajectory space
$ \mathcal {H} ^ + $) if for any set $B\subset \mathcal {H} ^ + $ which is bounded in $ L_\infty (0, +
\infty; E) $,  there is $h \geq 0 $
such that for all $t \geq h $:
$$ T (t) B\subset P. $$ \end{definition}

\begin{definition} A set $ \mathcal {U}\subset $ $C ([0, + \infty ); E_0) \cap L_\infty (0,
+ \infty; E) $ is called the \textit{minimal trajectory attractor} (for
the trajectory space $ \mathcal {H} ^ + $) if

i) $ \mathcal {U} $ is compact in $C ([0, + \infty); E_0) $ and bounded in
$L_\infty (0, + \infty; E) $;

ii) $T (t) \mathcal {U}=  \mathcal {U}$ for any $t \geq 0 $;

iii) $ \mathcal {U} $ is attracting in the sense of Definition 4.11;

iv) $ \mathcal {U} $ is contained in any other set satisfying conditions i), ii), iii).
\end{definition}

\begin{definition} A set $ \mathcal {A} \subset E $ is called the \textit{global
attractor} (in $E_0 $) for the trajectory space $ \mathcal {H}
^ + $ if

i) $ \mathcal {A} $ is compact in $E_0 $ and bounded in $E $;

ii) for any bounded in $ L_\infty (0, + \infty; E) $ set $B\subset
\mathcal {H} ^ + $ the attraction property is fulfilled:
$$\sup\limits _ {u\in B} ^ {} \inf\limits _ {v\in \mathcal {A}} ^ {} \|u (t)-v \| _ {E_0} \underset {t\to\infty} {\to} 0 $$

iii) $ \mathcal {A} $ is the minimal set satisfying conditions i)
and ii) (that is, $ \mathcal {A} $ is contained in every set
satisfying conditions i) and ii)). \end{definition}

\begin{theorem}(see \cite[Corollary 4.2.1, Lemma 4.2.9]{gruy}) Assume that there exists an absorbing set $P $ for
the trajectory space $ \mathcal {H} ^ + $, which is relatively compact in $C ([0, + \infty); E_0) $
and bounded in $L_\infty (0, + \infty; E) $. Then there exists a
minimal trajectory attractor $ \mathcal {U} $ for the trajectory
space $ \mathcal {H} ^ + $. \end{theorem}

The structure of minimal trajectory attractors is discussed in \cite[Chapter 4]{gruy} (see also \cite{cvbook}). In particular, the minimal trajectory attractor contains the set of those solutions to the considered problem that can be continued to the whole
real axis being uniformly bounded in $E $ and continuous with
values in $E_0 $; however, some extra elements may appear. The structure of global attractors is determined by the structure of minimal trajectory attractors:

\begin {theorem} (see \cite[Theorem 4.2.2]{gruy}) If there exists a minimal
trajectory attractor $ \mathcal {U} $ for the trajectory space $
\mathcal {H} ^ + $, then there is a global attractor $ \mathcal
{A} $ for the trajectory space $ \mathcal {H} ^ + $, and for all $t\geq 0 $ one has $ \mathcal {A} =\{y (t) |y\in\mathcal {U} \}. $
\end {theorem}

\section{Attractors for the polymeric diffusion problem}

We are going to construct the minimal trajectory attractor
and the global attractor for problem (3.9)-(3.11). Let us choose $L_2 (\Omega)^2 $ as the space $E$ and the space
$H ^ {-\delta} (\Omega) ^2$ as the space $E_0$, where $ \delta \in (0,1] $ is a fixed
number. The trajectory space $ \mathcal {H}
^ + $ is the set of all weak solutions to (3.9)-(3.11) from class (3.13). It is contained in $L_\infty (0, + \infty; E) $ due to Lemma 5.1 (below) and in  $C ([0, + \infty); E_0) $  (clearly).


Dissipativity of problem (3.9)-(3.11) is stated by the following result:

 \begin{lemma} Let $(v,\tau)$ be a weak solution to (3.9)-(3.11). Then \be\frac {\nu D} 2 \|v(t)\|^2  +\frac E 2 \|\tau(t)+\nu v(t)\|^2+ \frac \mu 2 \|v(t)\|^2_{-1}\leq $$ $$ e^{-\gamma t}\left(\frac {\nu D} 2 \|v(0)\|^2  +\frac E 2 \|\tau(0)+\nu v(0)\|^2+\frac \mu 2 \|v(0)\|^2_{-1}\right)+ \Gamma \ee for all $ t>0$, where $\Gamma$ and $\gamma$ are some fixed positive numbers, which are independent of $v,\tau$ and $t$. \end{lemma}

\begin{proof} The pair $(v,\varpi=\tau+\nu v)$ satisfies (3.6), (3.7). Take the
$H^{-1}(\Omega)$-scalar product of (3.6) and $\mu v +\nu v'$ and the
"bra-ket"
 of (3.7) and $E\varpi$ for a.a. $t\in (0,\infty)$, and
add the results (cf. proof of Theorem 3.2): \be \frac \mu 2 \frac{d \|v\|^2_{-1}}{d t}+ \nu\|v'\|^2_{-1}+ \mu D\|v\| ^2 + \frac {\nu D} 2\frac{d \|v\|^2}{d t}  +\frac E 2 \frac{d \|\varpi\|^2}{d t}$$ $$+ E (\beta (x, v,\varpi) \varpi, \varpi) =(h, \mu v +\nu v')_{-1} + E(g (x,v,\varpi), \varpi).\ee

Hence, \be \frac \mu 2 \frac{d \|v\|^2_{-1}}{d t}+ \nu\|v'\|^2_{-1}+ \mu D\|v\| ^2 + \frac {\nu D} 2\frac{d \|v\|^2}{d t}  +\frac E 2 \frac{d \|\varpi\|^2}{d t}$$ $$+ E\beta_G \|\varpi\|^2 \leq C(\| v\|_{-1} +\|v'\|_{-1} + \|\varpi\|)\leq C(\| v\| +\|v'\|_{-1} + \|\varpi\|).\ee
The Cauchy inequality for scalars can be written in the form $C\eta \leq \epsilon \eta^2 + \frac {C^2}{4 \epsilon} .$ Thus, (5.3) implies \be \frac \mu 2 \frac{d \|v\|^2_{-1}}{d t}+  \frac {\mu D }2 \|v\| ^2 + \frac {\nu D} 2\frac{d \|v\|^2}{d t}  +\frac E 2 \frac{d \|\varpi\|^2}{d t}+ \frac {E\beta_G} 2 \|\varpi\|^2 \leq C.\ee

Let $\chi(t)=\frac \mu 2 \|v(t)\|^2_{-1}+ \frac {\nu D} 2 \|v(t)\|^2  +\frac E 2 \|\varpi(t)\|^2.$ Obviously, $$\frac {\mu D }2 \|v(t)\| ^2 + \frac {E\beta_G} 2 \|\varpi(t)\|^2  \geq \gamma \chi(t)$$ for some $\gamma>0$. Thus, $\chi'(t)+ \gamma \chi(t) \leq C$, so  $\chi(t)\leq e^{-\gamma t}\chi(0)+ (1+\gamma^{-1})C$ by \cite[Lemma II.1.3]{cvbook}, and (5.1) follows. \end{proof}

The main result of this section is \begin{theorem} The trajectory space $ \mathcal {H} ^ + $ possesses a minimal trajectory attractor and a global attractor. \end{theorem} \begin{proof} Due to Theorems 4.15 and 4.16, it suffices to find an absorbing set for
the trajectory space $ \mathcal {H} ^ + $, which is relatively compact in $C ([0, + \infty); E_0) $
and bounded in $L_\infty (0, + \infty; E) $. Consider the set $P$ of weak solutions to (3.9)-(3.11) which satisfy the estimate \be\frac {\nu D} 2 \|v(t)\|^2  +\frac E 2 \|\tau(t)+\nu v(t)\|^2+ \frac \mu 2 \|v(t)\|^2_{-1}\leq 2 \Gamma, \forall t\geq 0. \ee It is an absorbing set for
the trajectory space $ \mathcal {H} ^ + $ and is
bounded in $L_\infty (0, + \infty; E) $. By (3.9), the set $\{v',\ (v,\tau)\in P\}$ is bounded in $L_{\infty} (0, + \infty; H^{-2}(\Omega)) $. Moreover, (3.10) yields that $\{\tau',\ (v,\tau)\in P\}$ is bounded in $L_\infty (0, + \infty; L_2(\Omega)) $. The embedding $E \subset E_0$ is compact. By \cite[Corollary 4]{sim1}, the set $\{y|_{[0,M]}, y\in P\}$ is relatively compact in $C ([0, M]; E_0) $ for any $M>0$.  This implies (cf. \cite[p. 183]{gruy}) that $P$ is relatively compact in $C ([0, + \infty); E_0) $. \end{proof}




\begin {thebibliography} {99}

\bibitem{am1} Amann, H. \textit{Global existence for a class of highly degenerate parabolic systems},
Japan J. Indust. Appl. Math. \textbf {8} (1991), 143--151.
\bibitem{am2} Amann, H. \textit{Highly degenerate quasilinear parabolic
systems}, Ann. Scuola Norm. Sup. Pisa Cl. Sci. \textbf {18} (1991),
135--166.
\bibitem {vib}  Babin, A.V., Vishik, M.I.  \textit{Attractors of evolution
equations}, Nauka, Moscow, 1989 (Eng. transl.: North-Holland, 1992).

\bibitem{ball} Ball, J.M. \textit{Continuity properties and global attractors of
generalized semiflows and the Navier-Stokes equations}, J. Nonlinear Sci. \textbf {7} (1997), 475--502.
\bibitem {vis2} Chepyzhov, V.V., Vishik, M.I. \textit{Evolution equations and their trajectory
attractors}, J. Math Pures Appl. \textbf{76} (1997), 913--964.
\bibitem{cvbook} Chepyzhov, V.V., Vishik, M.I. \textit{Attractors for Equations of Mathematical Physics},
AMS Colloquium Publications 49, Providence, RI, 2002.
\bibitem{chn0} Cohen, D.S., White, A.B., Jr. \textit{Sharp fronts due to diffusion and viscoelastic relaxation
in polymers}, SIAM J. Appl. Math. \textbf {51} (1991), 472--483.

\bibitem{chn1} Cohen, D.S., White, A.B., Jr., Witelski, T.P. \textit{Shock
formation in a multidimensional viscoelastic diffusive system},
SIAM J. Appl. Math. \textbf {55} (1995), 348--368.


\bibitem{dur} Durning, C. J. \textit{Differential sorption in viscoelastic fluids}, J. Polymer Sci., Polymer Phys. Ed. \textbf{23} (1985),
1831--1855.

\bibitem{ed} Edwards, D.A. \textit{A mathematical model for trapping skinning in
polymers}, Studies in Applied Mathematics \textbf{99} (1997), 49--80.
\bibitem{ed2} Edwards, D.A.  \textit{A spatially nonlocal model for polymer desorption}, Journal of Engineering Mathematics \textbf{53} (2005), 221--238.
\bibitem{cc} Edwards, D.A., Cairncross, R.A. \textit{Desorption overshoot in polymer-penetrant systems: Asymptotic and
computational results}, SIAM J. Appl. Math. \textbf{63} (2002), 98--115.
\bibitem{chn3} Edwards, D.A., Cohen, D.S. \textit{A mathematical model for a dissolving polymer}, AIChE J. \textbf{18} (1995), 2345--2355.

\bibitem{hu} Hu, B. \textit{Diffusion of penetrant in a polymer: a free boundary problem}, SIAM J. Math. Anal.
\textbf{22} (1991), 934--956.
\bibitem{bei} Hu, B., Zhang, J. \textit{Global existence for a class of non-Fickian polymer-penetrant
systems}, J. Partial Diff. Eqs. \textbf{9} (1996), 193--208.
\bibitem{kras} Krasnoselskii, M. \textit{Topological methods in the theory of nonlinear integral equations},
Gostehizdat, 1956 (Russian); Engl. transl., Macmillan, 1964.


\bibitem{lee} Lee, Sang-Wha. \textit{Relaxation Characteristics of Poly(vinylidene fluoride) and Ethylene-chlorotrifluoroethylene
in the Transient Uptake of Aromatic Solvents}, Korean J. Chem. Eng. \textbf{21} (2004), 1119--1125.

\bibitem{moi} Moise, I., Rosa, R., Wang, X.M. \textit{Attractors for noncompact semigroups via energy equations}, Nonlinearity \textbf{11} (1998), 1369--1393.
\bibitem{riv} Riviere, B., Shaw, S. \textit{Discontinuous Galerkin finite element approximation of nonlinear non-Fickian diffusion in viscoelastic polymers}, SIAM Journal on Numerical Analysis \textbf{44} (2006), 2650--2670.
\bibitem{sel} Sell, G. \textit{Global attractors for the three-dimensional
Navier-Stokes equations}, J. Dyn. Diff. Eq. \textbf {8} (1996),
1--33.
\bibitem{sim1}  Simon, J. \textit{Compact sets in the space $L^p(0,T; B)$}, Ann. Mat. Pura Appl. \textbf{146} (1987), 65--96.

\bibitem{skry}  Skrypnik, I. V. \textit{Methods for analysis of nonlinear elliptic boundary value problems}, Translations of Mathematical Monographs 139, Amer. Math. Soc., 1994.



\bibitem{tw} Thomas, N., Windle, A.H. \textit{Transport of methanol in poly-(methyl-methocry-late)}, Polymer \textbf{19} (1978), 255--265.

\bibitem{tw1} Thomas, N., Windle, A.H. \textit{A theory of Case II diffusion}, Polymer \textbf{23} (1982), 529--542.

\bibitem{diss} Vorotnikov, D.A., \textit{Dissipative solutions for equations of viscoelastic diffusion
in polymers}, J. Math. Anal. Appl. \textbf{339} (2008), 876--888.

\bibitem{var} Vorotnikov, D.A., \textit{Weak solvability for equations of viscoelastic diffusion in
polymers with variable coefficients}, J. Differential
Equations 
\textbf{246} (2009), 1038--1056.

\bibitem{iter} Vorotnikov, D.A. \textit{On iterating concentration and periodic regimes at the anomalous diffusion in polymers}, submitted.

\bibitem{nova} Vorotnikov, D.A. \textit{The second boundary value problem for equations of
viscoelastic diffusion in polymers}, In: Differential Equations: Systems, Applications and Analysis, F. Columbus (ed.), Nova Science, in preparation.

\bibitem{jmfm} Vorotnikov, D.A., Zvyagin, V.G. \textit{Trajectory and global attractors of the boundary
value problem for autonomous motion equations of viscoelastic
medium}, J. Math. Fluid Mech. \textbf{10} (2008), 19--44.

\bibitem{wit} Witelski, T.P. \textit{Traveling wave solutions for case II diffusion in polymers}, Journal of Polymer Science: Part B: Polymer Physics \textbf{34} (1996), 141--150.


\bibitem{gruy} Zvyagin, V.G., Vorotnikov, D.A. \textit{Topological approximation methods for evolutionary problems of nonlinear hydrodynamics}, de Gruyter Series in Nonlinear Analysis and Applications 12, Walter de Gruyter \& Co., Berlin, 2008. 

\end {thebibliography}
\end {document}